\documentclass[a4paper,11 pt]{amsart}

\usepackage{amssymb}  
\usepackage{mathtools}      
\usepackage{mathabx}        
\usepackage[bb=fourier,cal=euler,scr=rsfs]{mathalfa}	
\usepackage{enumitem}       
\usepackage[colorlinks=true,backref=page]{hyperref} 

\usepackage[ansinew]{inputenc}

\renewcommand*{\backref}[1]{}
\renewcommand*{\backrefalt}[4]{\quad \tiny
  \ifcase #1 (\textbf{NOT CITED.})%
  \or    (Cited on page~#2.)%
  \else   (Cited on pages~#2.)%
  \fi}

\hypersetup{bookmarksdepth = 3} 
\numberwithin{equation}{section}     

\setlist[enumerate,1]{label={\upshape(\roman*)},ref=\roman*}
\setlist[enumerate,2]{label={\upshape(\alph*)},ref=\alph*}

\newcommand{\R}{\mbox{$\mathbb{R}$}}

 \def\ZZ{{\mathbb Z}}

    \def\cS{\mathcal{S}}
    
    \def\cU{\mathcal{U}}
\def\cD{\mathcal{D}}    
    \def\cW{\mathcal{W}}
\def\cF{\mathcal{F}}  \def\cL{\mathcal{L}}

\newtheorem*{teo*}{Theorem}

\newtheorem{teo}{Theorem}[section]
\newtheorem{addendum}[teo]{Addendum}

\newtheorem{fact}[teo]{Fact}
\newtheorem*{af}{Claim}
\newtheorem{lema}[teo]{Lemma}
\newtheorem{prop}[teo]{Proposition}

\newtheorem{thmintro}{Theorem}

\theoremstyle{definition}
\newtheorem{defi}{Definition}[section]
\theoremstyle{remark}
\newtheorem{obs}[teo]{Remark}

\newcommand{\eps}{\varepsilon}

\title[Incoherence in Seifert manifolds]{Dynamical incoherence for a large class of partially hyperbolic diffeomorphisms}

\author[T.~Barthelm\'e]{Thomas Barthelm\'e}
\address{Queen's University, Kingston, ON}
\email{thomas.barthelme@queensu.ca}
\urladdr{sites.google.com/site/thomasbarthelme}

\author[S.~Fenley]{Sergio R.\ Fenley} 
\address{Florida State University, Tallahassee, FL 32306 }
\email{fenley@math.fsu.edu}

\author[S.~Frankel]{Steven Frankel} 
\address{Washington University in St Louis, St Louis, MO}
\email{steven.frankel@wustl.edu}

\author[R.~Potrie]{Rafael Potrie} 
\address{Centro de Matem\'atica, Universidad de la Rep\'ublica, Uruguay}
\curraddr{Institute for Advanced Study, Princeton, NJ 08540, USA}
\email{rpotrie@cmat.edu.uy}
\urladdr{http://www.cmat.edu.uy/~rpotrie/}

\keywords{Partial hyperbolicity, 3-manifold topology, foliations, classification.}
\subjclass[2010]{37D30,57R30,37C15,57M50,37D20}
\begin{document}

\begin{abstract}
We show that if a partially hyperbolic diffeomorphism of a Seifert manifold induces a map in the base which has a pseudo-Anosov component then it cannot be dynamically coherent. This extends \cite{BGHP} to the whole isotopy class. We relate the techniques with the study of certain partially hyperbolic diffeomorphisms in hyperbolic 3-manifolds performed in \cite{BFFP}. The appendix reviews some consequences of the Nielsen-Thurston classification of surface homeomorphisms to the dynamics of lifts of such maps to the universal cover. 
\end{abstract}

\maketitle

\section{Introduction}
A diffeomorphism $f: M \to M$ of a closed manifold is said to be \emph{partially hyperbolic} if there is a continuous $Df$-invariant splitting $TM = E^s \oplus E^c \oplus E^u$ into non-trivial bundles  and a constant $\ell >0$ such that for every unit vectors $v^\sigma \in E^\sigma(x)$ ($\sigma = s,c,u$) one has that: 

$$ \|Df^\ell v^s \| < \min \{1 , \|Df^\ell v^c \| \} \ \text{ and } \   \|Df^\ell v^u \| > \max \{1 , \|Df^\ell v^c \| \}.  $$

These diffeomorphisms arise naturally in the study of robust dynamical properties (because it is a $C^1$-open property). In addition, they capture quite well the properties of important examples, since, e.g. most homogeneous dynamics and several geometrically defined dynamical systems verify these properties. We refer the reader to \cite{HP-survey,PotICM} for examples and motivation for their study. 

An important feature of the study of the dynamics of partially hyperbolic diffeomorphisms is the possibility of dimension reduction. In a nutshell, one considers the strong bundles $E^s$ and $E^u$ as \emph{well understood} (see e.g. \cite{HPS} for the unique integrability of the bundles and their dynamics) and tries to understand the dynamics along the center direction. For this purpose, it is very useful when one can integrate the center bundle into an $f$-invariant foliation. 

There are several possible ways for the center bundle to be integrable which are discussed in detail in \cite{BurnsWilkinson}.  There it is also explained why the following definition arises as the natural integrability condition to ask a partially hyperbolic diffeomorphism: We say that a partially hyperbolic diffeomorphism $f$ is \emph{dynamically coherent} if there exist $f$-invariant foliations $\cW^{cs}$ and $\cW^{cu}$ tangent to $E^{cs}=E^s \oplus E^c$ and $E^{cu}= E^c \oplus E^u$ respectively. 

There are also many ways a partially hyperbolic diffeomorphism may fail to be dynamically coherent. The first one was observed by Wilkinson (see \cite{BurnsWilkinson}) and reference therein) and has to do with the potential failure of the Frobenius bracket condition on the center direction. 

However, the question of dynamical coherence for partially hyperbolic diffeomorphisms with one dimensional center remained open for quite a while. One needs to notice here that what may fail in this case is the unique integrability due to the lack of smoothness of the bundles. The first non-dynamically coherent examples where presented in \cite{HHU-noncoherent}. In 3-dimensional manifolds with sufficiently small fundamental group we have now a quite good understanding of dynamical coherence (see \cite{HP-survey} for a full account).

More recently, in \cite{BGHP} new examples of (robustly) non-dynamically coherent partially hyperbolic diffeomorphisms where constructed. These present several new features, since they can be made transitive and even absolutely partially hyperbolic (a concept we shall not define here, but for which we refer the reader to \cite{BGHP}). Even if the constructions in \cite{BGHP} are rather flexible, the proof of non-dynamical coherence in that paper depends on some very specific choices of the example and a quite precise control on the bundles. In this paper we provide a new proof of the non-dynamical coherence of these examples which extends to \emph{every} partially hyperbolic diffeomorphism in the same isotopy class of those constructed in \cite{BGHP}. 

The main result is the following: 

\begin{thmintro}\label{teo.main}
Let $f: M \to M$ be a partially hyperbolic diffeomorphism in a Seifert 3-manifold with hyperbolic base. Assume that the induced action of $f$ in the base has a pseudo-Anosov component, then $f$ is not dynamically coherent.  
\end{thmintro}

We will say that a 3-manifold is Seifert with hyperbolic base, if it admits a finite cover which is a circle bundle over a surface of genus $g \geq 2$. This is not the standard definition, but since to show Theorem \ref{teo.main} we are allowed to take a finite cover this is a convenient definition. (We refer the reader to \cite{HaPS} and references therein for other definitions and equivalences of a 3-manifold being Seifert with hyperbolic base.)

The fact that every map on such a circle bundle is homotopic to a fiber preserving one is standard (see eg. \cite{He}), and this explains what we mean by the induced action on the base (in this case, up to finite cover, on the surface). We will not define having a pseudo-Anosov component here, but we will state a result which is what we need to use about this in Theorem \ref{t.goodliftsurfacemap}.  We will give a more detailed treatment with the important references in Appendix \ref{appendixA}. 

Notice that in \cite{BFFP} (see also \cite{BFFP-announce}) we have shown that if $f$ is a partially hyperbolic diffeomorphism in a Seifert manifold $M$ which is homotopic to the identity, then $f$ must be dynamically coherent (and we even give a full classification of such maps). In \cite{BGHP} it is shown that if $f$ in a Seifert manifold induces the identity in the base, then it has to be homotopic to the identity. This leaves open the case of the examples which first appeared in \cite{BGP} where the action in the base is a Dehn-twist. It will become clear from the techniques (and the results in the appendix) why this case cannot be treated with the ideas we use to prove Theorem \ref{teo.main}. 

We end the introduction by saying that even if independent of \cite{BFFP}, this paper shares several ideas with one case which is treated in that paper, which we call double translation in hyperbolic manifolds (see in particular \cite[Sections 7 and 8]{BFFP}). In section \ref{s.further} we give some positive results in the direction of understanding these partially hyperbolic maps along the lines of what is done in \cite[Section 12]{BFFP}. The link  between these two situations is given by the following fact: in hyperbolic 3-manifolds there is a well developed theory (see  \cite{Thurston,CalegariPA,Fen2002}) of dynamics transverse to certain foliations which contains several features of the pseudo-Anosov dynamics that we use here in the setting of Seifert manifolds to induce dynamics from the ideal boundary to the universal cover of the manifold.

\medskip
\medskip

{\small \emph{Acknowledgements:} We thank A. Gogolev for pointing out the reference \cite{Fathi}.  T.Barthelm\'e was partially supported by the NSERC (Funding reference number RGPIN-2017-04592). S.Fenley was partially supported by Simons Foundation grant numbers 280429 and 637554. S. Frankel was partially supported by National Science Foundation grant number DMS-1611768. Any opinions, findings, and conclusions or recommendations expressed in this material are those of the authors and do not necessarily reflect the views of the National Science Foundation. R. Potrie was partially supported by CSIC 618, FCE-1-2017-1-135352. This work was completed while R.P. was serving as a Von Neumann fellow at IAS, funded by Minerva Research Fundation Membership Fund and NSF DMS-1638352.}

\section{Some preliminaries}\label{s.prelim}
The proof of Theorem \ref{teo.main} will be by contradiction. So we will assume that $f: M \to M$ is a dynamically coherent partially hyperbolic diffeomorphism of a closed Seifert 3-manifold with hyperbolic base and with $f$ acting on the base with a pseudo-Anosov piece (see Appendix \ref{appendixA}). Since a finite cover of $M$ makes it a circle bundle over a orientable surface, and iterates and lifts to finite covers of dynamically coherent partially hyperbolic diffeomorphisms are dynamically coherent we can make the following standing assumptions that will hold in sections \ref{s.compactset}, \ref{s.periodiccs} and \ref{s.teoA}:

\begin{itemize}
\item $M$ is a circle bundle over a closed orientable surface $S$ of genus $g \geq 2$. 
\item $f$ is a dynamically coherent partially hyperbolic diffeomorphism of $M$ with orientable bundles and such that $Df$ preserves their orientation.  We denote by $\cW^{cs}$ and $\cW^{cu}$ the center stable and center unstable foliations respectively. 
\item $f$ induces an automorphism $\rho: \pi_1(S) \to \pi_1(S)$ which has a pseudo-Anosov component. 
\end{itemize}

The last condition can be understood just because the homotopy class of fibers are in the center of $\pi_1(M)$ and therefore $f_\ast: \pi_1(M) \to \pi_1(M)$ preserves the fiber and so induces an automorphism $\rho$ of $\pi_1(S)$. Having a pseudo-Anosov component can be read by this automorphism (see \cite{Thurston-surfaces,Gilman}). We refer to the appendix \ref{appendixA} for more information on this; below, in Theorem \ref{t.goodliftsurfacemap} we state what we need from this information. 

First, we need a very classical property of partially hyperbolic diffeomorphisms (see \cite{HPS}). For a foliation $\cF$ we denote by $\cF(x)$ the leaf through the point $x$. 

\begin{teo}[Stable manifold Theorem] 
Let $f: M \to M$ be a partially hyperbolic diffeomorphism of a closed manifold $M$. Then, the bundles $E^s$ and $E^u$ are uniquely integrable into foliations $\cW^s$ and $\cW^u$ such that if $y \in \cW^s(x)$ (resp. $y \in \cW^u(x)$) one has that $d(f^n(x),f^n(y)) \to 0$ exponentially fast as $n \to +\infty$ (resp. $n \to -\infty$). 
\end{teo}

We state a result from \cite{HaPS} which allows to reduce to the case where the foliations are horizontal

\begin{teo}[\cite{HaPS}] 
Let $f$ be a dynamically coherent partially hyperbolic diffeomorphism of a circle bundle over a surface. Then, the foliations  $\cW^{cs}$ and $\cW^{cu}$ are horizontal. 
\end{teo}

This means that one can modify the fibration in order that every fiber is transverse to the leaves of $\cW^{cs}$ and $\cW^{cu}$. In particular, one can lift a hyperbolic metric on the base surface to the leaves to have a leafwise hyperbolic metric on the leaves. (We notice however that one can not in principle assume that the same fibration is transverse to both foliations simultaneously. For this reason, from now on we will work only with the center stable foliation $\cW^{cs}$, of course, symmetric statements also hold for $\cW^{cu}$ after changing the fibration.)

Let $\widetilde M$ be the universal cover of $M$. Let $\delta \in \pi_1(M)$ be a deck transformation of $\widetilde M$ associated to the center of $\pi_1(M)$ which corresponds to the homotopy class of a fiber with a given orientation. Since (up to iterate) $f$ preserves $\delta$ it follows that $f$ lifts to the intermediate cover $\hat M = \widetilde M/_{<\delta>}$. Clearly, there are several lifts of $f$. 

By putting a hyperbolic metric in $S$, and lifting the metric to the leaves of $\cW^{cs}$ we can identify $\hat M \cong \mathbb{H}^2 \times S^1$.  The foliation $\cW^{cs}$ lifts to a foliation $\hat \cW^{cs}$ for which the canonical projection $p : \mathbb{H}^2 \times S^1 \to \mathbb{H}^2$ into the first coordinate is a isometry. Notice that in this lift, given a point $x \in \mathbb{H}^2$ and a leaf $L \in \hat \cW^{cs}$  one has that $p^{-1}(x) \cap L$ is a unique point. We denote by $d_{\mathbb{H}^2}$ to the hyperbolic distance in $\mathbb{H}^2$ and for $L \in \hat \cW^{cs}$  as $d_L$ the induced distance in the leaf. 

In particular, one gets that the \emph{leaf space} of this foliation (i.e. the quotient space $\cL^{cs}= \hat M /_{\hat \cW^{cs}}$) is a topological circle. The strategy of the proof of Theorem A by contradiction is to show that the action of a nice lift $\hat f$ of $f$ in $\cL^{cs}$ is everywhere expanding by using the fact that transversally to $\hat \cW^{cs}$ we have the unstable foliation, this will give that there must be periodic points in $\cL^{cs}$ and all of them should be expanding, a contradiction. However, to be able to perform the argument, we need to show that there are points in every leaf of $\hat \cW^{cs}$ of $\hat M$ which do not escape to infinity by application of $\hat f$ since otherwise the expansion could be produced by the holonomy of the foliation\footnote{ A nice example to understand this is the time one map of the geodesic flow in negative curvature, for which there is a lift where every leaf of the center foliation is fixed, but this does not contradict partial hyperbolicity.}. To get such compact invariant sets in $\hat M$ we will choose specific lifts of $f$ to $\hat M$ which have some specific actions in the boundary at infinity that will allow us to control the dynamics in the interior of each leaf. 

Notice if one takes any lift $\hat f : \hat M \to \hat M$, since $f$ is isotopic to a fiber preserving map, one can define a projection $F: \mathbb{H}^2 \to \mathbb{H}^2$ which is the lift of a homeomorphism $h$ of $S$ induced by the fiber preserving map homotopic to $f$ (notice that $h$ is not canonically defined, but it is well defined up to homotopy which is what we care about). Since the projection $p$ restricted to each leaf of $\hat \cW^{cs}$ is a isometry we deduce that: 

\begin{fact}\label{f.bdistance}
There is a uniform $K_0>0$ so that for every $x \in \hat M$, if $L = \hat f (\cW^{cs}(x))$ one has that: 
\[ | d_{\mathbb{H}^2} (p(x), F(p(x))) - d_{L} (\hat f(x), p^{-1}(p(x))\cap L) | < K_0 . \]
\end{fact}

Said otherwise, one can compare the action that a given lift of the mapping class of $f$ induces on $\mathbb{H}^2$ with  the action of the corresponding lift $\hat f$ on the leaves of $\hat \cW^{cs}$. Notice that if $\hat f$ preserves a center stable leaf $L$ then the second distance becomes $d_L(\hat f(x), x)$. 

What we need to know about maps of surfaces having a pseudo-Anosov component is the following: 

\begin{teo}\label{t.goodliftsurfacemap}
Let $h: S \to S$ be a surface homeomorphism having a pseudo-Anosov component where $S$ is a hyperbolic closed surface. Then, there exist a lift $\tilde h: \mathbb{H}^2 \to \mathbb{H}^2$ to the universal cover $\tilde S \sim \mathbb{H}^2$ such that:
\begin{enumerate}
\item The lift $\tilde h$ extends to $E=\partial_{\infty} \mathbb{H}^2$ as a circle homeomorphism which has exactly four fixed points, two attracting and two repelling. 
\item If we denote by $\cD = \mathbb{H}^2 \cup E$ the Gromov compactification of $\mathbb{H}^2$ one gets a a homeomorphism (also called) $\tilde h : \cD \to \cD$ of the closed disk so that: 
\begin{enumerate}
\item The attracting points in $E$ verify that they have a basis of neighborhoods $U_n$ in $\cD$ such that $d_{\mathbb{H}^2}(\partial U_n, \tilde h(U_n)) \to \infty$. The symmetric statement holds for repelling points looking at $\tilde h^{-1}$. 
\item For every $K>0$ there is a compact set $D \subset \mathbb{H}^2$ such that for every $y \notin D$ one has that $d_{\mathbb{H}^2}(y, \tilde h(y)) > K$. 
\end{enumerate}
\end{enumerate}
\end{teo} 

See the appendix \ref{appendixA} for a proof as well as further information about surface homeomorphisms that may be helpful in understanding further properties of partially hyperbolic diffeomorphisms on Seifert manifolds. The notion in item ($i$)($a$) has to do with the notion of \emph{super attracting/repelling} points we introduce there (cf. Definition \ref{def.stronglyattracting}).This statement can also be found implicit in \cite[Section 9]{Miller} or in \cite{HT}. The difference with the case where the pseudo-Anosov component is the whole surface is that in that case one knows that \emph{every} lift looks more or less as the one given by the previous theorem (with possibly a different number of attracting and repelling points).  

\section{Constructing an invariant compact set in a nice cover}\label{s.compactset}
Consider $f: M \to M$ as in the standing assumptions of section \ref{s.prelim}. In particular, we can assume that the circle bundle $M$ over the surface $S$ verifies that fibers are transverse to $\cW^{cs}$. As explained in the previous section, $f$ is homotopic to a homeomorphism preserving the fibers of the fibration and this induces a homeomorphism $h: S\to S$. By hypothesis, we know that $h$ has a pseudo-Anosov component, so that Theorem \ref{t.goodliftsurfacemap} applies. 

We will therefore choose once and for all a lift $\hat f$ of $f$ to $\hat M$ such that the induced map $F$ in $\mathbb{H}^2$ (as in Fact \ref{f.bdistance}) verifies the conditions ensured by Theorem \ref{t.goodliftsurfacemap}.   

The main result of this section is the following (compare with \cite[Proposition 8.1]{BFFP}). 

\begin{prop}\label{p.compactset} 
There is a compact $\hat f$-invariant set $T_{\hat f} \subset \hat M$ which intersects every $L \in \hat \cW^{cs}$. 
\end{prop}

\begin{proof}
Notice first that  if we prove the result for an iterate $\hat f^k$ of $\hat f$ then the result will still hold for $\hat f$ since the union of the first $k-1$-iterates of the compact set one finds will give the desired set. 

Consider the map $F$ acting on $\mathbb{H}^2$. As explained in Theorem \ref{t.goodliftsurfacemap} the map $F$ extends to $\cD = \mathbb{H}^2 \cup E$ continuously. The action of $F$ on $E$ has four fixed points, two attracting and two repelling which are alternating. Denote by $\ell_a$ to the geodesic in $\mathbb{H}^2$ joining the points $\ell_a^{\pm}$ in $E$ which are attracting by $F$ and $\ell_r$ to the geodesic joining the respective repelling points which we denote by $\ell_r^{\pm}$. Since the points are alternating, it follows that $\ell_a$ and $\ell_r$ intersect transversally in a point $x_0 \in \mathbb{H}^2$. 

Now, fix $K_0$ given by Fact \ref{f.bdistance} and let $K_1 \gg 10K_0$.  

Let $D_1 \subset \mathbb{H}^2$ be a compact set given by Theorem \ref{t.goodliftsurfacemap} for $F$ and the constant $K_1$. The set $D_1$ verifies the following properties (see figure \ref{f.31}):

\begin{itemize}
\item For every $y \notin D_1$ one has that $d_{\mathbb{H}^2}(y, F(y))> K_1$. 
\item There are open neighborhoods $U^{+}$ and $U^-$ of $\ell_a^+$ and $\ell_a^-$ in $\cD$ respectively such that $F(\overline{U^{\pm}}) \subset U^{\pm}$ so that $d_{\mathbb{H}^2} (\partial U^\pm \cap \mathbb{H}^2, F(U^{\pm})) \gg K_1$ and such that $U^\pm \cap D_1$ is a connected non-empty set. 
\item Symmetrically, there are  open neighborhoods $V^{+}$ and $V^-$ of $\ell_r^+$ and $\ell_r^-$ in $\cD$ respectively such that $F^{-1}(\overline{V^{\pm}}) \subset V^{\pm}$ so that $d_{\mathbb{H}^2} (\partial V^\pm \cap \mathbb{H}^2, F(V^{\pm}) \gg K_1$ and such that $V^\pm \cap D_1$ is a connected non-empty set.
\end{itemize}

Once $D_1$ and the neighborhoods $U^{\pm}$ and $V^{\pm}$ have been chosen, one can consider $k>0$ large enough so that we know that $F^k$ acting on the boundary verifies that the end points of $F^k(V^{\pm})$ are contained in the closure of $U^{-}\cup U^{+}$ in $\cD$.  Once $k$ has been fixed with this property, one can consider new constants $K_0'$ by Fact \ref{f.bdistance} applied to $\hat f^k$ and $K_1' \gg 10K_0'$ and choose $D$ a larger disk (containing $D_1$) so that: 

\begin{itemize}
\item For every $y \notin D$ one has that $d_{\mathbb{H}^2}(y, F^k(y))> K_1'$.
\item $D \cup U^{-} \cup U^+$ contains the $K_1'$ neighborhood of $F(\partial V^{\pm})$ with respect to the hyperbolic distance.
\end{itemize}

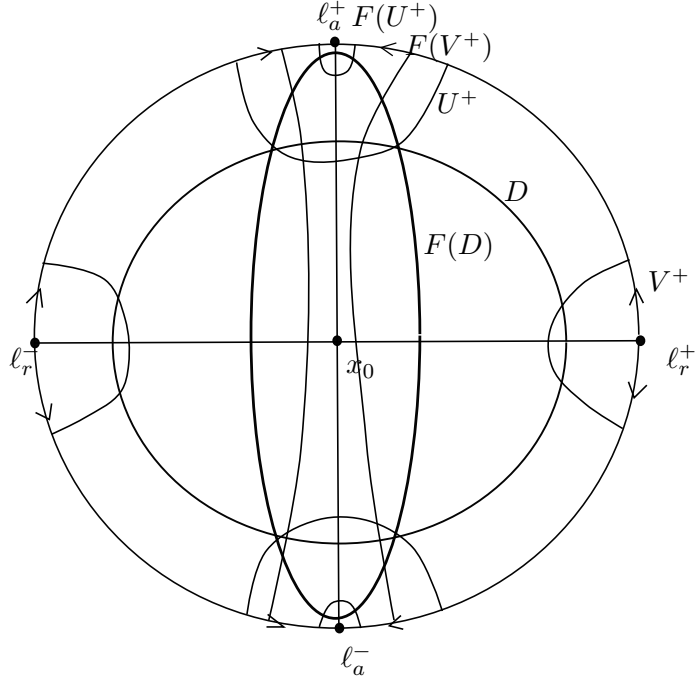
\begin{figure}[ht]
\begin{center}
\includegraphics[scale=0.75]{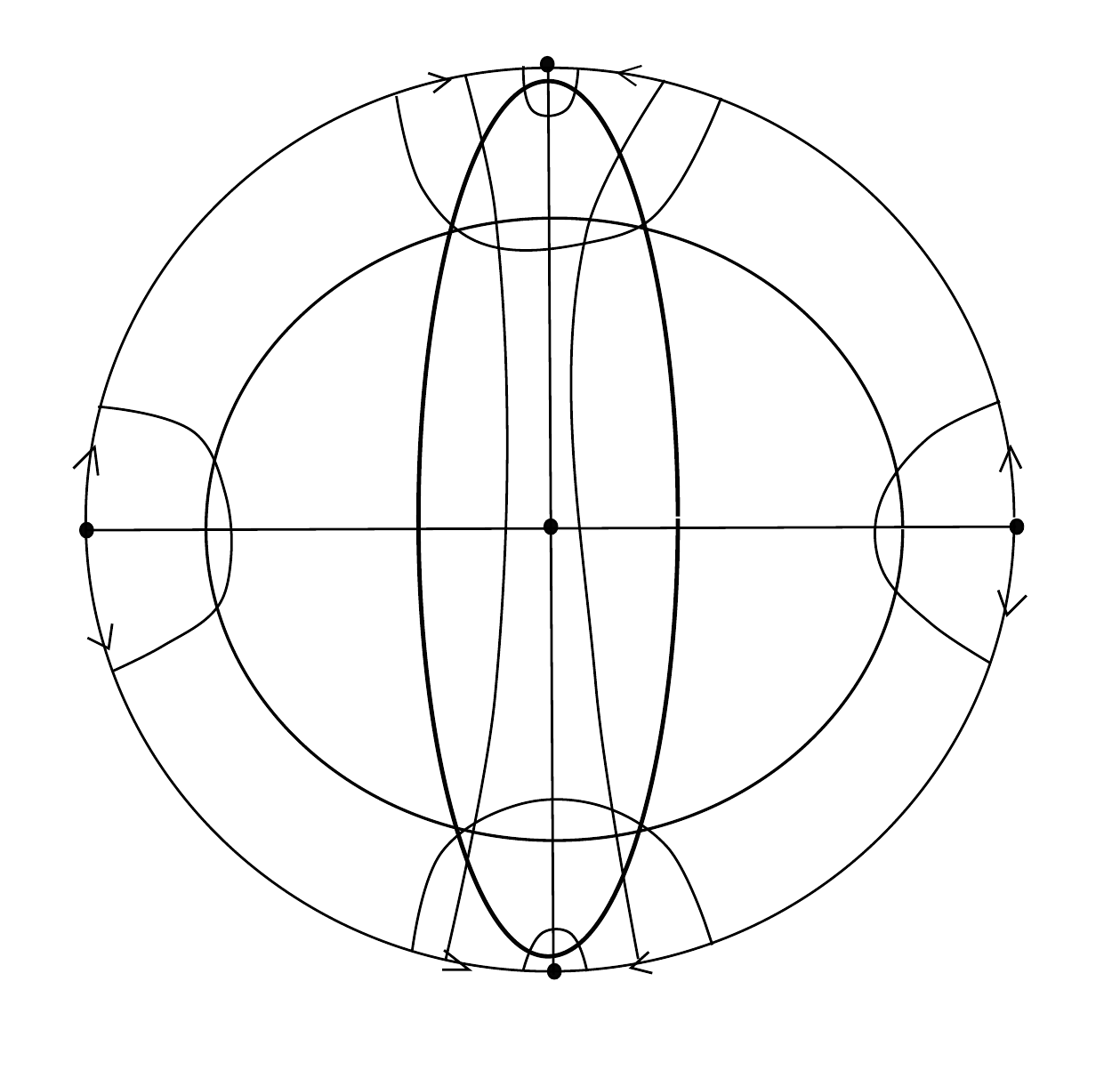}
\begin{picture}(0,0)
\put(-19,122){$\ell_r^+$}
\put(-139,120){$x_0$}
\put(-150, 252){$\ell_a^+$}
\put(-141,9){$\ell_a^-$}
\put(-265,122){$\ell_r^-$}
\put(-80,185){$D$}
\put(-110,165){$F(D)$}
\put(-105,218){$U^+$}
\put(-117,240){$F(V^+)$}
\put(-137,251){$F(U^+)$}
\put(-26, 151){$V^+$}
\end{picture}
\end{center}
\vspace{-0.5cm}
\caption{{\small Depiction of the dynamics of $F^k$.}}\label{f.31}
\end{figure} 

Now, let $\hat D = p^{-1}(D \setminus (V^+ \cup V^-))$. We will show that the compact $\hat f^k$-invariant set 

$$ T_{\hat f^k} = \bigcap_{n \in \ZZ} \hat f^{kn} (\hat D) , $$ 

\noindent is non-empty and intersects every $L \in \hat \cW^{cs}$. This will complete the proof of the proposition as explained above. Let us denote as $\hat D_L$ to $\hat D \cap L$ for every $L \in \hat \cW^{cs}$. Similarly, we denote by $U_L^{\pm}$ to $p^{-1}(U^{\pm} \cap \mathbb{H}^2) \cap L$ and $V_L^{\pm} = p^{-1}(V^{\pm} \cap \mathbb{H}^2) \cap L$. 

The key point is the following. For notational simplicity we will assume from now on that $k=1$ which makes no changes in the argument.

\begin{af} 
For every $R \subset \hat D_L$ which is compact connected and intersects both $U_L^+$ and $U_L^-$, one has that $\hat f(R) \cap \hat D_{\hat f(L)}$ contains a compact connected set which intersects $U_{\hat f(L)}^+$ and $U_{\hat f(L)}^-$ and is disjoint from $V_{\hat f(L)}^{\pm}$. 
\end{af} 

\begin{proof}
This is quite direct from the choice of $\hat D$ and Fact \ref{f.bdistance}. 
\end{proof}

Now, let $R^n_L = \bigcap_{j=0}^n \hat f^j( \hat D_{\hat f^{-j}(L)})$. The previous claim shows that $R^1_L$ contains a compact connected set intersecting both $U_L^+$ and $U_L^-$ for every $L$. We will prove inductively that the same is true for $R^n_L$. So, assume that $R^{n-1}_L$ contains a compact connected set intersecting both $U_L^+$ and $U_L^-$ for every $L \in \hat \cW^{cs}$. Then, the same will be true for $R_L^n= \hat f(R^{n-1}_{\hat f^{-1}(L)}) \cap \hat D_L$ by using the induction hypothesis for $\hat f^{-1}(L)$ and the claim . 

If $R^\infty_L$ is $\bigcap_{n\geq 0} R^n_L$ we deduce that it contains a compact connected set joining $U_L^+$ and $U_L^-$. Similarly\footnote{Clearly, the choice of $k$ could be different for future and backward iterates, but e.g. one can take the minimum common multiple of both as $k$ and continue.}, one can construct a set $Y^\infty_L = \bigcap_{n\leq 0}  \hat f^j( \hat D_{\hat f^{-j}(L)})$ and show it contains a compact connected set joining $V_L^+$ to $V_L^-$. These sets must intersect, and their intersection in $L$ is exactly $T_{\hat f} \cap L$. This completes the proof. 

\end{proof}

There are some simplifications with respect to \cite[Section 8]{BFFP} because on the one hand we do not treat the $p$-prong case (with $p\geq 3$) and because the coordinates given by the Seifert fibration allow to obtain some estimates in a more straightforward fashion. The reader should be able prove the analogous result for the $p$-prong case exactly as in \cite{BFFP} once the analogy exposed in this section is clear. We note that a similar statement follows from \cite{Fathi} but cannot be applied directly since we need that the invariant set intersects every leaf of $\hat \cW^{cs}$. 

\section{Finding periodic center-stable leaves}\label{s.periodiccs}
Using the set $T_{\hat f}$ and the transversal expansion along $\hat \cW^u$ we will show that (compare with \cite[Proposition 9.1]{BFFP}):

\begin{prop}\label{p.periodic} 
There is at least one $L \in \hat \cW^{cs}$ and $n>0$ such that $\hat f^n(L)=L$. 
\end{prop}

\begin{proof}
Consider a point $x \in T_{\hat f}$ such that $\hat f^{n_j}(x) \to x$ for some $n_j \to \infty$. This exists by compactness of $T_{\hat f}$ (indeed there exists a minimal subset $X$ of $T_{\hat f}$ so that every point $x \in X$ has a dense orbit in $X$). 

Now, fix a local product structure box $B$ around $x$ where the foliation $\hat \cW^{cs}$ is of the form $B \cong D_0 \times (0,1)$ by a chart  $\varphi: B \to D_0 \times (0,1)$ where center stable plaques are sent into horizontals. One can choose $B$ so that a small unstable arc $I$ centered at $x$ maps by $\varphi$ to a transversal intersecting every plaque $D_0 \times \{t\}$. For large $n_j$ one has that there is also an unstable arc $J$ through $y=\hat f^{n_j}(x)$ which projects by $\varphi$ to a transversal intersecting every plaque. Since $\hat f^{-n_j}$ contracts exponentially the length of unstable arcs, it follows that for large $n_j$ we can assume that $\hat f^{-n_j}(J) \subset I$. It follows that there is a plaque $ \varphi^{-1}(D_0 \times \{t_0\})$ whose image by $\hat f^{n_j}$ intersects itself. This gives a periodic center stable leaf of period $n_j$. 
\end{proof}

\begin{obs}
Notice that this proof has used the fact that $f$ is dynamically coherent. However, it can be adapted to more general situations, see \cite[Section 11.8]{BFFP}. 
\end{obs} 

\section{Proof of Theorem \ref{teo.main}}\label{s.teoA}
As a consequence of the proof in the previous section we can extend to show the following (compare with \cite[Section 9]{BFFP}): 

\begin{prop}\label{p.transverseexpansion} 
For every $L \in \hat \cW^{cs}$ such that there is $k>0$ with $\hat f^k(L)=L$ it follows that there is a neighborhood $I$ of $L$ in $\cL^{cs}= \hat M/_{\hat \cW^{cs}}$ such that for every $L' \in I$ one has that $\hat f^{kn}(L') \to L$ as $n \to -\infty$. 
\end{prop}

\begin{proof}
By a Lefschetz index argument (see eg. \cite[Appendix I]{BFFP}) one can see that there must be a fixed point $x \in T_{\hat f} \cap L$ by $\hat f^k$. Now, the result follows as in the proof of Proposition \ref{p.periodic} by considering a foliated chart around $x$ and using that $\hat f^k$ is expanding in the local unstable manifold of $x$. 
\end{proof}

\begin{obs}
This proof does use dynamical coherence crucially, as collapsing of center stable leaves can make the transverse behaviour change, e.g. if there is an interval of leaves collapsing at the fixed point $x \in L$. See section \ref{s.further} for more details. 
\end{obs}

This allows to complete the proof of Theorem \ref{teo.main}. 

\begin{proof}[Proof of Theorem \ref{teo.main}]
The proof is by contradiction. We have assumed that $f$ was dynamically coherent and constructed $\hat f$ acting on $\cL^{cs}$ as a circle homeomorphism. In Proposition \ref{p.periodic} we showed that $\hat f$ acting in $\cL^{cs}$ has periodic points. Proposition \ref{p.transverseexpansion} shows that \emph{every} periodic point of $\hat f$ in $\cL^{cs}$ is expanding. This is impossible, and gives a contradiction that completes the proof. 
\end{proof}

\section{Further properties and discussions}\label{s.further} 
The classification of partially hyperbolic diffeomorphisms in dimension 3 was started in the seminal paper of Bonatti and Wilkinson \cite{BW}. The fundamental results of Brin, Burago and Ivanov (see \cite{BI}) provided the main tool for the study of these systems by showing that even when a partially hyperbolic diffeomorphism is not dynamically coherent, it nevertheless (up to a finite lift and iterate) preserves a pair of invariant geometric structures tangent to $E^{cs}$ and $E^{cu}$ which can be sometimes treated as foliations. These objects are called \emph{branching foliations} and we decided not to introduce them in this short paper. We refer the reader to \cite{BI} or \cite{HP-survey} for an overview. 

In the same way as in \cite[Part 2]{BFFP} one can obtain many of the results that we stated without the assumption of dynamical coherence. The main difference is that one will not be able to get a contradiction (which is good, since there are examples!). One important consequence is that one can obtain an analog of Proposition \ref{p.compactset}.  Its proof follows the same scheme using branching foliations instead of foliations. We will state it without proof trusting that the interested reader can deduce it by looking at \cite[Part 2]{BFFP}:

\begin{prop}\label{p.compactset2}
Let $f: M \to M$ be a partially hyperbolic diffeomorphism of a closed circle bundle over a higher genus surface and with $f$ acting on the base with a pseudo-Anosov piece. Consider a lift $\hat f$ to $\hat M$ associated to a lift of the associated map $h: S\to S$ having $\geq 4$ periodic points alternatingly contracting and expanding at infinity.  Denote by $p: \hat M \to \tilde S$ the projection along the fibers. Then, there is a large disk $D$ in $\tilde S$ such that if $\hat D = p^{-1}(D)$ then we have that the set $T_\gamma= \bigcap_{n \in \ZZ} \hat f^n(\hat D)$ is non empty and essential in the sense that every non homotopically trivial curve in $\partial \hat D$ is homotopically non-trivial in $\hat M \setminus T_\gamma$.  
\end{prop}

We comment also on some properties that can be obtained if one finds some surfaces periodic by $f$ and tangent to $E^{cs}$. A more detailed treatment in a similar situation appears in \cite[Sections 14]{BFFP}. 

To state a result in this direction, we must change slightly the context we have been working with. Let $f: M \to M$ be a partially hyperbolic diffeomorphism of a circle bundle over a surface $S$. As before, $f$ induces a map $h: S \to S$ which we assume has a pseudo-Anosov component. We consider a lift $\hat f$ to $\hat M$ induced by a lift $\tilde h$ of $h$ as given by Theorem \ref{t.goodliftsurfacemap}. We will say that a ray $\eta: [0,\infty) \to \hat M$  is \emph{contracting fixed ray} if $\hat f (\mathrm{Im}(\eta)) = \mathrm{Im}(\eta)$, the unique fixed point of $\hat f$ in $\mathrm{Im}(\eta)$ is $\eta(0)$ and every other point $y \in \mathrm{Im}(\eta)$ verifies that $\hat f^n (y) \to \eta(0)$. 

\begin{prop}\label{p.landingbound}
Assume that $\hat f$ preserves a complete surface $\cS$ uniformly transverse\footnote{To be precise, what we need is that the metric induced by restricting the metric of $\hat M$ to $\cS$ is quasi-isometric to the metric pulled back by the projection to $\mathbb{H}^2$. Moreover, by inspection of the proof it is only needed that the surface is transverse to the fibers in a neighborhood of the ray.} to the fibers of $p: \hat M \cong \mathbb{H}^2 \times S^1 \to \mathbb{H}^2$ and contains a contracting fixed ray. Then, the projection of $\eta$ to $\mathbb{H}^2$ can be extended continuously to a closed arc $\overline \eta: [0, \infty] \to \mathbb{H}^2 \cup E$ and $\overline \eta(\infty)$ is one of the repelling points at infinity. 
\end{prop}

\begin{proof} Let $\mathcal{I}$ be the closure of $p(\eta([0,\infty))$ in $\mathbb{H}^2 \cup E$ and $I = \mathcal{I} \cap E$ which is a connected subset of $I$.  Clearly, $\tilde h(I)=I$ because $\eta$ is $\hat f$ invariant. 

We claim that $I$ cannot contain an attracting point $a$ of $\tilde h$ acting on $E$. If that were the case, one can consider an open set $U$ of $a$ in $\mathbb{H}^2 \cup E$ such that $\tilde h(\overline U) \subset U$. Because of the properties in Theorem \ref{t.goodliftsurfacemap} one can choose $U$ so that points of $\partial U \cap \mathbb{H}^2$ are mapped by $\tilde h$ with distance larger than $K_0$ which is much larger than the distance between $\tilde h$ and the induced dynamics of $\hat f$ on $\mathbb{H}^2$ via the projection of its dynamics on $\cS$ (see Fact \ref{f.bdistance}). This implies that $p(\eta([0,\infty))$ cannot intersect  $U$ because of its contracting behaviour by $\hat f$. 

It follows that $I$ must be a connected subset of $E$ avoiding neighborhoods of each attracting point. This implies that $I$ is a repelling point of $\tilde h$ in $E$. 
\end{proof}

Notice that this proof uses crucially that one can extend the dynamics in the boundary to some information in the interior of $\mathbb{H}^2$. This is not the case in general, in particular when $f$ has no pseudo-Anosov components, or when one chooses lifts of maps which are not pseudo-Anosov (see appendix \ref{appendixA} for more details on this).  Notice that on the other hand, in \cite{BFFP} the case of maps homotopic to the identity was successfully treated in Seifert manifolds. 
\appendix

\section{Lifts of surface maps}\label{appendixA}
In this appendix we study the dynamics of lifts of surface maps to the universal cover. Most of these results follow directly from the Nielsen-Thurston classification theory and are implicit in the standard references \cite{Thurston-surfaces, FLP, HT, Casson,Calegari}. Since we were not able to find the precise statements we will give an indication of the proofs assuming some familiarity with the theory of geodesic laminations. See also \cite{Gilman,Miller} for more details on the approach by Nielsen which is more or less what we follow here, based on \cite{HT}. For completeness, we have also treated the Dehn twist case which is not used in this paper. 

Similar results in a strictly more general setting of flows transverse and regulating to $\R$-covered foliations were obtained in \cite{Thurston,CalegariPA,Fen2002}, but only fully developed in the case of atoroidal manifolds which correspond to the pseudo-Anosov case. 

\subsection{Preliminaries}
Let $S$ be a closed orientable surface of genus $g\geq 2$ and let $f: S \to S$ a homeomorphism. We fix a hyperbolic metric on $S$ and identify the universal cover  $\tilde S$ of $S$ with $\mathbb{H}^2$. The deck transformation group $\Gamma = \pi_1(S)$ acts by isometries of $\mathbb{H}^2$. 

We denote by $E = \partial \tilde S = \partial_{\infty} \mathbb{H}^2 = \partial_\infty \Gamma$ the circle at infinity (or Gromov boundary) of $\tilde S$, and $\cD = \tilde S \cup E$ with the topology given by the Gromov compactification.  

A lift $g$ of $f$ to the universal cover $\tilde S$ of $S$ has a unique continuous extension to $\cD$  (see e.g. \cite[Corollary 1.2]{HT}), still denoted by $g: \cD \to \cD$ . This restricts to a homeomorphism $\hat g= g|_{E} : E \to E$ of the boundary circle. Two lifts of $f$ to the universal cover differ by a deck transformation and the dynamics of different lifts may vary. However, the action at infinity only depends on the homotopy class of the lift. 

\begin{fact}[see e.g. Corollary 1.4 in \cite{HT}] 
If two homeomorphisms $f,f'$ in a surface have lifts $g,g'$ so that $\hat g = g|_{E}= g'|_{E}=\hat g'$ then $f$ is homotopic to $f'$. Conversely, if two maps are homotopic and one considers a lift of the homotopy, then the action in $E$ is constant along the homotopy. 
\end{fact}

For an element $\gamma \in \Gamma$ we denote by $A(\gamma)$  the \emph{axis} of $\gamma$, that is, the unique  geodesic in $\mathbb{H}^2$ fixed by $\gamma$ oriented from the repelling fixed point $\gamma^-$ in $E$ to the attracting fixed point $\gamma^+$ in $E$.  In general, if $s$ is an oriented geodesic in $\mathbb{H}^2$, we denote by $s^-$ and $s^+$ the forward and backward endpoints of the geodesic in $E$. By a \emph{geodesic} in $D$ we mean the closure of a geodesic in $\mathbb{H}^2$ which we denote by  $\hat s = s \cup s^- \cup s^+$. 

Given $\hat s$ an oriented geodesic in $\cD$ we denote by $g_\ast \hat s$ to the (oriented) geodesic joining the extreme points of $\hat g(s^-)$ with $\hat g(s^+)$.  

For $\gamma \in \Gamma$ we denote by $C(\gamma)$  the geodesic in $S$ which is the projection of $A(\gamma)$. Notice that the lift of $C(\gamma)$ to the universal cover consists on all the geodesics $A(\eta)$ where $\eta$ is in the conjugacy class of $\gamma$. For a closed geodesic $c$ in $S$,  we denote as $f_\ast(c)$ the unique closed geodesic homotopic to $f(c)$, or equivalently, to the projection of $g_\ast(\tilde c)$ where $\tilde c$ is any geodesic which projects onto $c$. 

It follows from the Nielsen-Thurston classification that every surface homeomorphism is either periodic, reducible or pseudo-Anosov. This means $f$ is homotopic to a map $f_0$ that preserves a collection of simple closed geodesics $\{c_1, \ldots, c_d \}$. Note that there is a power $k>0$ such that $f_0^k(c_i)=c_i$ for all $i$. After cutting $S$ along the curves in the collection one has a collection $S_1, \ldots, S_\ell$ of subsurfaces fixed by $f_0^k$. Up to considering a larger $k$ there are two possibilites for $f_0^k$ acting on each  $S_i$ (c.f. \cite[Section 2]{HT}):

\begin{itemize}
\item $f_0^k$ induces the identity in the fundamental group of $S_i$ or, 
\item $f_0^k$ does not fix any simple closed geodesic in the interior of $S_i$. 
\end{itemize}

We will say that $f$ has a \emph{pseudo-Anosov piece} if $f_0^k$ acts as in the second possibility in some subsurface $S_i$ of the decomposition. In this second case, we call $S_i$ a pseudo-Anosov piece or component of $f$. The results in the Nielsen-Thurston classification theorem  provide specific representatives for $f$ in the pseudo-Anosov pieces but we will not need this here. We refer the reader to the cited references above for proofs and more detailed statements. 

In this appendix we will only be concerned with properties of the lifts of $f$ that are true up to homotopy and we want information of lifts of $f$ up to iterate. For this reason we will assume throughout that $f$ fixes all the surfaces of the decomposition $\{S_j\}$ of $S$ and such that in each surface it either induces identity or does not fix any simple closed curve of the interior. (Compare with \cite[Section 2]{HT}.)

This easily implies that we can assume that $\hat g$ has a non-empty set of fixed points in $E$.  

We want to understand the structure of fixed points of $\hat g$ on $E$ and how this affects the dynamics in the interior of the disk (i.e. on $\tilde S$). For this, it is useful to classify fixed points of $\hat g$.  

\begin{obs}\label{rem-nofix} Notice that if a point is not fixed by $\hat g$ on $E$ then it already forces the dynamics in the interior of $\cD$. The crucial point is that there is a uniform constant $C>0$ such that the image $g(c)$ of a geodesic $c$ by $g$ is in a $C$-neighborhood\footnote{To see this, one can assume that $f$ is a diffeomorphism (by approximation) and so the derivative of $f$ is uniformly bounded. Then, the image of a geodesic by $g$ is a quasigeodesic with constants that only depend on $f$. So, by the Morse Lemma it follows that it is a uniform distance appart of the geodesic joining the endpoints.} of the geodesic $g_\ast(c)$. So, if a point $x \in E$ verifies that $\hat g(x)\neq x$ then it follows that it has a neighborhood $U$ in $\cD$ such that $g(U) \cap U = \emptyset$. This motivates the study of fixed points of $\hat g$. 
\end{obs}

We say that $x \in E$ is an \emph{attracting} fixed point of $\hat g$ in $E$ if there is an open neighborhood $U$ of $x$ such that $\hat g(\overline U) \subset U$ and such that $\{x\}= \bigcap_n \hat g^n(U)$. A fixed point is \emph{repelling} if it is attracting for $\hat g^{-1}$. 

\begin{obs} It is important to consider a stronger version of attraction as attracting behavior  on $E$ does not imply attracting behaviour on $\cD$. A simple example can be made by choosing a map homotopic to the identity via a homotopy $f_t$ with $f_0=\mathrm{id}$, $f_1(x_0)=x_0$ and such that the closed curve $c=\{f_t(x_0)\}_{t \in [0,1]}$ is homotopically non-trivial. Then, if one considers the lift $g$ of $f_1$ obtained by lifting the homotopy from the lift of $f_0$ consisting on a deck transformation $\gamma$ associated to $c$, it follows that $g$  acts in $E$ as the deck transformation $\gamma$ (i.e. with an attracting and a repelling fixed point) but there is a sequence of lifts of $x_0$ which accumulates both fixed points of $\hat g$ and which are fixed by $g$. 
\end{obs}

\begin{defi}[Strongly attracting/repelling fixed points]\label{def.stronglyattracting} We say that a fixed point $x \in E$ is a \emph{strongly attracting} point if  for every $K>0$ and open interval $x \in I$ there is an open interval $U$ containing $x$ in its interior such that $\hat g(\overline U) \subset U$ and such that the hyperbolic distance from the geodesic joining the endpoints of $U$ and of $\hat g(U)$ is larger than $K$. A \emph{strongly repelling} point is a strongly attracting point for $\hat g$. 
\end{defi}

Notice that the definition only depends on the knowledge of $\hat g$ and not of $g$. 

An important (simple) consequence is the following: 

\begin{lema}
If $x$ is a strongly attracting fixed point for $\hat g$, then $x$ is an attractor for $g$ acting on $\cD$, that is, there is a neighborhood $\cU$ of $x$ in $\cD$ such that $g(\overline \cU) \subset \cU$ and such that $\{x\} = \bigcap_{n \geq 0} g^n(\cU)$. 
\end{lema}

The proof is immediate from the fact that $g$ maps a geodesic $c$ to curve a curve $g(c)$ which is at (uniformly) bounded distance from the geodesic $g_\ast (c)$ as remarked above. 

\subsection{Statement of results}

The first relevant result is implicit in \cite[Lemma 1.3]{HT}:

\begin{teo}\label{teo.boundarygeneral} If $\mathrm{Fix}(\hat g) \cap \mathrm{Fix}(\gamma) \neq \emptyset$ for some $\gamma \in \Gamma \setminus \{\mathrm{id}\}$ then $g$ commutes with $\gamma$ and $\mathrm{Fix}(\gamma) \subset \mathrm{Fix}(\hat g)$. In this case there are two possibilities: 
\begin{itemize}
\item either $\mathrm{Fix}(\hat g)=\mathrm{Fix}(\gamma)$ and the fixed points of $\hat g$ are not strongly attracting nor repelling, 
\item or $\mathrm{Fix}(\hat g) \neq \mathrm{Fix}(\gamma)$ in which case, for every $\eta \in \Gamma$ such that $\mathrm{Fix}(\eta) \cap \mathrm{Fix}(\hat g) \neq \emptyset$ it follows that $\mathrm{Fix}(\eta)$ is contained in the set of accumulation points of $\mathrm{Fix}(\hat g)$. 
\end{itemize}
Moreover, if $f$ is not homotopic to the identity then $\mathrm{Fix}(\hat g)$ has empty interior. 
\end{teo}


To get more precise information, we want to separate between the cases where $f$ contains a pseudo-Anosov part or not. In the case where it does not have pseudo-Anosov components, what we obtain is the following: 

\begin{teo}\label{teo.boundaryDehn} 
Let $f$ be a surface homeomorphism which is not homotopic to the identity and such that it has no pseudo-Anosov piece. Assume moreover that $g: \cD \to \cD$ is a lift of $f$ which has fixed points in $E$. If there are more than two fixed points of $\hat g$, then there exists $\gamma \in \Gamma$ such that $\mathrm{Fix}(\gamma) \subset \mathrm{Fix}(\hat g)$; moreover, the set of fixed points of $\hat g$ is a Cantor set and there is a dense set of $\mathrm{Fix}(\hat g)$ consisting of fixed points of $\eta \in \Gamma$. Otherwise $\hat g$ has exactly two fixed points which can be either strongly attracting and repelling or $\hat g$ commutes with some deck transformation. \end{teo}

When there is a pseudo-Anosov piece, one gets a different behaviour. 

\begin{teo}\label{teo.boundaryPA} 
Let $f$ be a surface homeomorphism with a pseudo-Anosov piece and let $g: \cD \to \cD$ be a lift of $f$. Then, up to taking an iterate, one of the following mutually exclusive possibilities holds: 
\begin{itemize}
\item either $\mathrm{Fix}(\hat g)$ consists of an even number of points which are alternatingly strongly attracting and repelling, or,
\item $\mathrm{Fix}(\hat g) \cap \mathrm{Fix}(\gamma) \neq \emptyset$ for some $\gamma \in \Gamma \setminus \{\mathrm{id}\}$. 
\end{itemize}
Moreover, up to taking iterates of $f$ there are always lifts which belong to the first class and have exactly four fixed points in $E$. 
\end{teo}

Notice that when $f$ is pseudo-Anosov in the whole surface $S$ (equivalently, if no simple closed curve is mapped to the same free homotopy class by some iterate of $f$) only the first possibility can occur. One can say more in each of the possibilities (e.g. in the first case one can characterize which lifts have four or more fixed points in $E$ and also show that some have exactly two fixed points). We state some further properties in the second case (which can only happen if $f$ is reducible): 

\begin{addendum}\label{addPA}
In the setting of Theorem \ref{teo.boundaryPA} the lifts corresponding to the second possibility have the following options: 
\begin{itemize}
\item there is a strongly attracting fixed point (in which case there is $\gamma \in \Gamma$ whose fixed points are also fixed by $\hat g$ accumulated from both sides by alternating strongly attracting and strongly repelling fixed points of $\hat g$),
\item $\sharp \mathrm{Fix}(\hat g) =2$ one attracting but not strongly attracting and one repelling but not strongly repelling,
\item the set of fixed points $\mathrm{Fix}(\hat g)$ is a cantor set, in particular, no fixed point is either attracting or repelling.
\end{itemize}
Moreover, there are always lifts with the first and second options and there are lifts with the third option only if $f$ has a periodic piece. 
\end{addendum}

\subsection{Outline of the proof}

Let us first quickly explain how to prove Theorem \ref{teo.boundarygeneral}.

\begin{proof}[Outiline of the proof of Theorem \ref{teo.boundarygeneral}]
The key points follow from \cite[Lemma 1.3]{HT}.  Let $\gamma$ be primitive. If one considers $\eta= g^{-1} \gamma g$ one gets a deck transformation that shares a fixed point with $\gamma$ and therefore $\eta = \gamma^k$ for some $k >0$ but since $\gamma$ is primitive one must have $k=1$. It follows that $g$ commutes with $\gamma$ which implies that $\gamma$ preserves the fixed point set of $\hat g$ and therefore this set being closed forces $\mathrm{Fix}(\gamma) \subset \mathrm{Fix}(\hat g)$. Moreover, if there are fixed points of $\hat g$ which are not fixed by $\gamma$, then these must accumulate on both $\gamma^+$ and $\gamma^-$.   

The fact that if $g$ commutes with a deck transformation $\gamma$ then the endpoints $\gamma^{\pm}$ cannot be strongly attracting or repelling is quite direct. 

Finally, it is easy to see that if the set of fixed points of $\hat g$ has non-empty interior, then $f$ has to be the identity as it induces identity in a sufficiently large set of closed geodesics in $S$. 
\end{proof} 

To prove the rest of the results, one important property is implicit in the proof of \cite[Lemma 3.1]{HT}. Compare with Remark \ref{rem-nofix}. 

\begin{lema}\label{l.boundarydyn}
Let $g : \cD \to \cD$ be a lift of a surface homeomorphism $f$.  Assume that there exists $x \in \mathrm{Fix}(\hat g)$ such that for some $K>0$ and every  neighborhood $U$ of $x$ in $\cD$ there is $y \in \tilde S \cap U$ such that $d_{\tilde S}(y, g(y))< K$. Then, for every neighborhood $U$ of $x$ there is $\gamma \in \pi_1(S)$ such that $g$ commutes with $\gamma$ and $A(\gamma)$ intersects $U$ and has $\gamma^+ \in U \cap E$.  
\end{lema}

\begin{proof} Take a sequence $y_n \in \tilde S$ such that $y_n \to x$ in $\cD$ so that $d_{\tilde S}(y_n, g(y_n)) <K$.  

Since a fundamental domain of $S$ and its image by $g$ have bounded diameter in $\tilde S$ we can assume that, up to changing $K$, one has that all $y_n$ are lifts of the same point $y_0$. That is we assume that $y_n = \gamma_n y_0$ for some $\gamma_n \in \pi_1(S)$. It is no loss of generality  to assume that all $\gamma_n$ are pairwise distinct. It follows that:

$$ d_{\tilde S} (\gamma_n^{-1} g \gamma_n y_0, y_0) \leq K \ \ ; \ \ \forall n. $$

And therefore for all $\eps>0$ there are $n\neq m$ so that $$d_{\tilde S}(\gamma_n^{-1} g \gamma_n y_0, \gamma_m^{-1} g \gamma_m y_0) < \eps$$ \noindent and since $\gamma_n^{-1} g \gamma_n$ are lifts of $f$ it follows that for small $\eps$ one has that $\gamma_n^{-1}g \gamma_n = \gamma_m^{-1} g \gamma_m$. This implies that $g$ commutes with the deck transformation $\gamma_n \gamma_m^{-1}$. 

Since one can choose $m-n$ arbitrarily large, one gets that $A(\gamma_n \gamma_m^{-1})$ has an endpoint close to $x$. This completes the proof. 
\end{proof}

As a consequence, we obtain:

\begin{lema}\label{l.boundarydyn2}
Let $g : \cD \to \cD$ be a lift of a surface homeomorphism $f$.  Assume that $g$ does not commute with any deck transformation. Then every fixed point of $\hat g$ is strongly attracting or strongly repelling.  
\end{lema}

\begin{proof}
Let $x$ be a fixed point of $\hat g$ in $E$ and let $r:[0,\infty] \to \cD$ be a geodesic ray landing on $x$.

By Lemma \ref{l.boundarydyn} we know that there is a basis of neighborhoods $\{U_K\}$ of $x$ in $\cD$ such that for every $y \in U_K \cap \tilde S$ the distance  $d_{\tilde S}(y, g(y))>K$. 

Consider a basis of neighborhoods $\{V_N\}$ of $x$ obtained by taking the geodesics orthogonal to $r$ at the point $r(N)$. It follows that there is a function $K: \mathbb{Z}_{>0} \to \mathbb{Z}_{>0}$ such that $V_N \subset U_{K(N)}$ and $K(N) \to \infty$ with $N$. 

Using the fact that the image by $g$ of a geodesic is a uniform quasigeodesic, one can then see that the image of the boundary of $V_N$ is a quasigeodesic more or less orthogonal to the quasigeodesic $g(r)$ ending in $x$ (one can compute the precise bounds for the angle using the constants of the quasigeodesics and some hyperbolic trigonometry). Since the distance between $g(r(N))$ and $r(N)$ is larger than $K(N)$ we get that the point has to be either super attracting or super repelling. 
\end{proof}

Let us now explain the proof of Theorem \ref{teo.boundaryDehn}: 

\begin{proof}[Outline of the proof of Theorem \ref{teo.boundaryDehn}] 
By assumption (which amounts to take $f$ up to homotopy and iterate), $f$ fixes all the curves in a collection $\{c_1, \ldots, c_k\}$ decomposing $S$ in a finite family $S_1, \ldots, S_\ell$ of closed surfaces such that are fixed by $f$ and such that $f$ induces the identity on the fundamental group each subsurface. We assume that the decomposition is minimal in the sense that one cannot glue two surfaces and have the same property. Since $f$ is not homotopic to the identity then  $k\geq 1$ (but note that $\ell$ still could be $1$ if the collection $c_1, \ldots, c_k$ is not separating).  

The lift $g$ of $f$ acts on a graph defined by putting vertices in each connected component of the lifts of the surfaces $\{S_i\}$ and an edge joining the vertices associated to subsurfaces which intersect in $\tilde S$. It is easy to see that this graph is a tree.

Assume that $g$ fixes some connected component $\tilde S_j$ of the lift of some subsurface $S_j$. Then, restricted to $\tilde S_j$, the map $g$ must be a bounded distance away from some deck transformation $\gamma$ fixing $\tilde S_j$ because $f$ induces the identity on $S_j$ and moreover $g$ commutes with $\gamma$. If $\gamma$ is the identity, it follows that the fixed point set of $\hat g$ is exactly the boundary at infinity of the subsurface $\tilde S_j$ which is a Cantor set. Since none of the fixed points can be strongly attracting or repelling Lemma \ref{l.boundarydyn} then implies that the situation claimed in the statement is verified. If $\gamma$ is not peripheral in $S_j$ this implies that the action of $\hat g$ resembles that of the dynamics of $\gamma$ at infinity in that it has exactly the same fixed points which are attracting and repelling (but not strongly attracting/repelling). If $\gamma$ is peripheral, one looks at the adjacent surface which is also fixed by $g$ and the same analysis gives the statement of Theorem \ref{teo.boundaryDehn}. 

If $g$ does not fix any such connected component, one gets an axis which must be represented by a geodesic in $\cD$ whose endpoints are the unique fixed points of $\hat g$. Notice that $g$ cannot commute with any deck transformation because it would correspond to a closed geodesic intersecting some of the the curves $c_i$ where $f$ acts as a Dehn-twist. Now, Lemma \ref{l.boundarydyn2} implies these fixed points must be strongly attracting and repelling.  
\end{proof}

We end with an outline of Theorem \ref{teo.boundaryPA} and its Addendum \ref{addPA}. 

\begin{proof}[Outline of the proof of Theorem  \ref{teo.boundaryPA} and Addendum \ref{addPA}] Up to homotopy and iterate, we will assume that $f$ preserves each of the subsurfaces $\{S_i\}$ (compare with \cite[Section 2]{HT}) defined above and is such that in each pseudo-Anosov piece preserves a pair of transverse geodesic laminations (see \cite[Sectiona 3 and 4]{HT}). We refer the reader also to \cite[Section 9]{Miller} for another treatment of this case. 
 
Let $S_i$ a surface where $f$ is pseudo-Anosov and let $\Lambda^s$ and $\Lambda^u$ be the geodesic laminations invariant by $f$ (see \cite[Section 3]{HT}). For $\tilde S_i$ a lift of $f$ we can take $g$ a lift fixing $\tilde S_i$ but which does not preserve any boundary component of $\tilde S_i$ in $\tilde S$. One can see that $g$ corresponds to a \emph{non-peripheral} Nielsen class of $f|_{S_i}$ (c.f. \cite[page 183]{HT}) and therefore has at least two attracting fixed points and two repelling fixed points at infinity and all fixed points at infinity are limit points of invariant curves of the lifted laminations $\Lambda^s$ and $\Lambda^u$. By Theorem \ref{teo.boundarygeneral} the map $g$ cannot commute with any deck transformation. Then, the action at those points is alternatingly strongly attracting and repelling by Lemma \ref{l.boundarydyn2}. It can be seen that it is impossible that all non-peripheral Nielsen classes have more than four boundary points as more than four points force some open polygonal region inside the lift of $\Lambda^s$ or $\Lambda^u$ and each such component occupies some hyperbolic area (see e.g. \cite[Lemma 4.1]{HT}) so at least one must have four fixed endpoints. This establishes the existence of a lift with the desired properties (and therefore a proof of Theorem \ref{t.goodliftsurfacemap}). 

Now, we must show that if $\mathrm{Fix}(\hat g) \cap \mathrm{Fix}(\gamma) = \emptyset$ for all $\gamma \in \Gamma$ then we must be in the first situation. If $g$ does not fix any $\tilde S_i$, the same argument as in Theorem \ref{teo.boundaryDehn} implies that there are exactly one super attracting and one supper repelling fixed point. If $g$ fixes some $\tilde S_i$, the hypothesis implies that in the projected surface $S_i$ the map $f$ is pseudo-Anosov. By assumption $g$ cannot preserve any boundary component of $\tilde S_i$. Then, one must consider two cases, either $g$ fixes some leaf of the lifted lamination $\tilde \Lambda^s$ of $\Lambda^s$. In this case, $g$ must have a fixed point in each of the fixed leaves, and therefore must preserve transverse leaves of the transverse lamination $\tilde \Lambda^u$ lifted of $\Lambda^u$. This gives an even number greater or equal to four fixed points in the boundary which are alternatingly strongly attracting or repelling (due to Lemma \ref{l.boundarydyn2}). 

If $g$ does not preserve any leaf of $\tilde \Lambda^s$ then it follows that it acts freely on the leaf space of the lamination which is a tree. Therefore, there is an invariant axis, and it follows that this defines one attracting and one repelling point at infinity in $E$ for the action of $g$. By Lemma \ref{l.boundarydyn2} these points must be strongly attracting and repelling. (We remark here that it is always possible to choose lifts with this property.) 

Finally, to get the addendum, let $\gamma \in \pi_1(S)$ such that $\mathrm{Fix}(\gamma) \subset \mathrm{Fix}(\hat g)$, then, by Theorem \ref{teo.boundarygeneral} we get that if the sets coincide the points cannot be strongly attracting or repelling. Otherwise, case one is obtained if $\gamma$ corresponds to the boundary of a surface where $f$ is pseudo-Anosov and $g$ corresponds to a lift fixing the peripheral Nielsen class (c.f \cite[page 184]{HT}). The last case is treated similar to Theorem \ref{teo.boundaryDehn}.
\end{proof}

\end{document}